\documentclass[11p,reqno]{amsart}
\usepackage{amsmath}
\usepackage{amsfonts}
\usepackage{amssymb}
\usepackage{graphicx}
\usepackage{color}
\newtheorem{theorem}{Theorem}[section]\newtheorem{thm}[theorem]{Theorem}
\newtheorem*{theorem*}{Theorem}
\newtheorem{lemma}{Lemma}[section]
\newtheorem{corollary}[theorem]{Corollary}

\newtheorem{prop}{Proposition}[section]
\newtheorem{remark}[theorem]{Remark}

\def \b {\beta}

\def\Ric{\text{Ric}}

\def\a{\alpha}
\def\l{\lambda}

\def\e{\epsilon}
\def\p{\partial}

\def\R{\Bbb R}

\def\Sph{\Bbb S}

\def\vp{\varphi}

\def\Ric{\operatorname{Ric}}

\newcommand{\eps}{{\varepsilon}}

\numberwithin{equation}{section}

\begin{document}

\title[]{Sharp lower bound for the first eigenvalue of the Weighted $p$-Laplacian}

\author{Xiaolong Li}
\address{Department of Mathematics, University of California, Irvine, Irvine, CA 92697, USA}
\email{xiaolol1@uci.edu}

\author{Kui Wang}\thanks{The research of the second author is supported by NSFC No.11601359} 
\address{School of Mathematical Sciences, Soochow University, Suzhou, 215006, China}
\email{kuiwang@suda.edu.cn}


\subjclass[2010]{35P15, 35P30}
\keywords{Eigenvalue estimates, weighted $p$-Laplacian, Bakry-Emery manifolds, and modulus of continuity estimates.}

\maketitle

\begin{abstract}

   We prove sharp lower bound estimates for the first nonzero eigenvalue of the weighted $p$-Lapacian operator with $1< p< \infty$ on a compact Bakry-Emery manifold $(M^n,g,f)$ satisfying $\Ric+\nabla^2 f \geq \kappa \, g$, provided that either $1<p \leq 2$ or $\kappa \leq 0$. Same conclusions hold when the manifold has nonempty boundary if we assume it is strictly convex and put Neumann boundary conditions on it.
   For $1<p \leq 2$, we provide a simple proof via the modulus of continuity estimates method.
  The proof for $\kappa \leq 0$ is based on a sharp gradient comparison theorem for the eigenfunction and a careful analysis of the underlying one-dimensional model equation.
  Our results generalize the work of Valtorta\cite{Valtorta12} and Naber-Valtorta\cite{NV14} for the $p$-Laplacian (namely $f=\text{const}$), and the work of Bakry-Qian\cite{BQ00} for the $f$-Laplacian (namely $p=2$).
\end{abstract}


\section{Introduction and Main Results}
Let $(M^n,g)$ be an $n$-dimensional compact Riemannian manifold. The eigenvalues of the Laplace-Beltrami operator $\Delta$, defined by
$$\Delta u := \frac{1}{\sqrt{\det g}}\frac{\partial }{\partial x_i} \left(\sqrt{\det g} g^{ij} \frac{\partial u}{\partial x_j} \right),$$
has been extensively studied for a long time.
Many results have been obtained by various authors under various hypotheses on curvature, diameter, and dimension.
Among the numerous beautiful results is the optimal lower bound on the first nonzero eigenvalue of the Laplacian operator in terms of the dimension $n$, the diamater $D$, and Ricci curvature lower bound $k$. Let $\l_1(M,g)$ be the first nonzero eigenvalue of the Laplacian, i.e.,
\begin{equation*}
    \l_1(M,g) =\inf \left\{ \frac{\int_M |\nabla u|^2 d\mu}{\int_M u^2 d\mu} : u \in W^{1,2}(M)\setminus\{0\},  \int_M u\, d\mu =0 \right\}.
\end{equation*}
\begin{thm}\label{Thm first eigenvalue for Laplacian}
If we let
\begin{equation*}
    \l_1(n,k,D)=\inf\{ \l_1(M,g) : M \text{ closed }, \text{dim}(M)=n, \Ric \geq (n-1)k, \text{diam}(M) \leq D \},
\end{equation*}
then
\begin{equation*}
    \l_1(n,k,D)=\mu(n,k,D),
\end{equation*}
where $\mu(n,k,D)$ is the first nonzero Neumann eigenvalue of the one-dimensional problem:
\begin{equation*}
  \vp''-(n-1)T_k \vp' =-\mu(n,k,D)\vp
\end{equation*}
on the interval $[-D/2,D/2]$, with $T_k$ given by
\begin{equation*}
     T_k(t)=\begin{cases}
   \sqrt{k} \tan{(\sqrt{k}t)}, & k>0, \\
   1, & k=0, \\
   -\sqrt{-k}\tanh{(\sqrt{-k}t)}, & k<0.
    \end{cases}
\end{equation*}
\end{thm}
The same holds in the larger class where $M$ is allowed to have a convex boundary $\p M$ and $u$ satisfies the Neumann boundary condition $\frac{\p u}{\p \nu}=0$ with $\nu$ being the inward normal along $\p M$. For $k=0$, this theorem was proved by Zhong and Yang\cite{ZY84}.
The result for general $k\in \R$ was obtained by Chen and Wang\cite{CY94}\cite{CY95} using probabilistic coupling method applied to the Brownian motion associated with heat flow, and independently by Kr\"oger\cite{Kroger98} via gradient estimates.
All of these methods have their roots in the arguments of Li\cite{Li79} and Li and Yau\cite{LY80}\cite{LY86}.
For $k>0$, this recovers Lichnerowicz's theorem (see for instance \cite[Theorem 5.1]{Libook}, 
in view of Myers' diameter bound
$D\leq \frac{\pi}{\sqrt{k}}$.
However, for smaller $D$, this is sharper.
It is worth mentioning that
a simple and elegant proof using the modulus of continuity estimates was given by Andrews and Clutterbuck\cite{AC13}(see also \cite{WZ17} for an elliptic proof based on \cite{AC13}).
The sharpness is demonstrated by constructing examples of Riemannian manifolds with given diameter bounds and Ricci curvature lower bounds such that the first nonzero eigenvalue is as close desired to $\mu(n,k,D)$, see for example\cite[Section 5]{AC13}.

The above mentioned results for Laplacian was encapsulated in a more general setting by Bakry and Qian\cite{BQ00}. Recall that a triple $(M,g,f)$, consisting of an $n$-dimensional Riemannian manifold $(M,g)$ and a potential function $f\in C^{\infty}(M)$, is called a Bakry-Emery manifold 
if 
\begin{equation*}
    \Ric+\nabla^2 f \geq \kappa \, g,
\end{equation*}
for some $\kappa \in \R$.
The tensor $\Ric +\nabla^2 f$, called the Bakry-Emery Ricci tensor, is a natural generalization of the classical Ricci tensor. It arises naturally in the study of diffusion processes, the Sobolev inequality, conformal geometry, and the Ricci flow. When the equality $\Ric+\nabla^2 f=\kappa \, g$ holds, the triple $(M^n,g,f)$ is called a shrinking ($\kappa >0$), or steady ($\kappa =0$), or expanding ($\kappa <0$) gradient Ricci soliton, respectively. Gradient Ricci solitons arise naturally in the singularity analysis of Ricci flow and they play a fundamental role in the study of Ricci flow. We refer the reader to 
\cite{CLNbook}\cite{Hamilton93}\cite{WW09}  
and the references therein for more discussions on Bakry-Emery manifolds and gradient Ricci solitons.
On Bakry-Emery manifolds or gradient Ricci solitons, it's more natural to look at the $f$-Laplacian (also called weighted Laplacian, Witten Laplacian, or drift Laplacian in the literature), which is defined by
\begin{equation*}
    \Delta_{f} u:= \Delta u- \langle \nabla u, \nabla f\rangle =e^f \text{div} (e^{-f} \nabla u).
\end{equation*}
This operator is self-adjoint with respect to the weighted measure $e^{-f}d\mu$, where $d\mu$ is the Riemannian measure induced by the metric $g$.
The first nonzero eigenvalue of $\Delta_f$, denoted by  $\l_{1,f}$, can be characterized in terms of a Poincar\'e inequality as the minimizer,
\begin{equation*}\label{lambda def}
    \lambda_{1,f} =\inf \left\{\frac{\int_M |\nabla u|^2 e^{-f} d\mu}{\int_M |u|^2 e^{-f} d\mu} : u \in W^{1,2} (M)\setminus\{0\}, \int_M u \, e^{-f} d \mu =0 \right\}.
\end{equation*}
The following sharp lower bound for $\l_{1,f}$ was proved by Bakry and Qian\cite{BQ00} (see also \cite{AN12} for a simple argument via the modulus of continuity estimates and the construction of examples to demonstrate the sharpness).
\begin{thm}\label{Thm BQ}
Let $\Omega$ be a compact manifold $M$, or a bounded strictly convex domain inside a complete manifold $M$, satisfying $\Ric +\nabla^2 f \geq \kappa \, g$ for some $\kappa \in \R$. Let $D$ be the diameter of $\Omega$. Then
\begin{equation*}
    \lambda_{1,f} \geq \mu(\kappa,D),
\end{equation*}
where $\mu(\kappa,D)$ is the first nonzero Neumann eigenvalue of the one-dimensional problem
\begin{equation}\label{ODE p=2}
    \vp'' -\kappa \, t \, \vp' =-\mu(\kappa, D) \vp
\end{equation}
on $[-D/2,D/2]$.
\end{thm}

In the past two decades, there has been a surge of interest in studying the first nonzero eigenvalues of  the $p$-Laplacian operator $\Delta_{p}$ and the weighted $p$-Laplacian operator $\Delta_{p,f}$, given by
\begin{equation*}
    \Delta_{p} u:= \text{div} (|\nabla u|^{p-2} \nabla u)
\end{equation*}
and
\begin{equation*}
    \Delta_{p,f}\, u:= \text{div} (|\nabla u|^{p-2} \nabla u) - |\nabla u|^{p-2} \langle \nabla u, \nabla f\rangle =e^f \text{div} (e^{-f}|\nabla u|^{p-2} \nabla u),
\end{equation*}
respectively.
Here $1<p<\infty$ and $u \in W^{1,p}(M)$, and the definitions are understood in the distributional sense. When $p=2$, the $p$-Laplacian operator reduces to the Laplacian, while the weighted $p$-Laplacian operator reduces to the $f$-Laplacian. Both operators are elliptic, but are singular for $1<p<2$ and degenerate for $p>2$.

Let $\l_p$ and $\l_{p,f}$ be the first nonzero eigenvalue of $\Delta_p$ and $\Delta_{p,f}$, respectively.
They can be characterized as
\begin{equation*}
    \l_{p} =\inf\left\{\frac{\int_M |\nabla u|^p d\mu}{\int_M |u|^p d\mu}: u \in W^{1,p}(M)\setminus\{0\}, \int_M |u|^{p-2}u \, d\mu=0 \right\},
\end{equation*}
and
\begin{equation*}\label{lambda def}
    \lambda_{p,f} =\inf \left\{\frac{\int_M |\nabla u|^p e^{-f} d\mu}{\int_M |u|^p e^{-f} d\mu} : u \in W^{1,p}(e^{-f}d \mu)\setminus\{0\}, \int_M |u|^{p-2} u \, e^{-f} d \mu =0 \right\}.
\end{equation*}
In case $\partial M \neq \emptyset$, the Neumann boundary condition $\frac{\partial u}{\partial \nu} =0$ is imposed. The following optimal lower bound for $\l_{p}$ has been established.


\begin{thm}\label{Thm p-Laplacian}
Let $M$ be a compact manifold (possibly with convex boundary) with diameter $D$ and $\Ric \geq (n-1) k$. Let $\l_{p}$ be the first nonzero eigenvalue of $(M,g)$ (with Neumann boundary condition if $\p M \neq \emptyset$). \\
(i). If $k>0$ and $\partial M =\emptyset$, then $$\l_{p}(M) \geq \l_{p}(\Sph^n(k)),$$
with equality holds if and only if $M$ is isometric to $\Sph^n(k)$,  the $n$-sphere with constant sectional curvature $k$. \\
(ii). If $k=0$, then
$$\lambda_{p} \geq (p-1)\left(\frac{\pi_p}{D}\right)^p,$$
where $\pi_p=\frac{2\pi}{p \sin{(\pi/p)}}$.
\\
(iii).If $k<0$, then
\begin{equation*}
    \l_{p} \geq \mu_{p}(n,k,D)
\end{equation*}
where $\mu_{p}(n,k,D)$ is the first nonzero Neumann eigenvalue of the one-dimensional problem
\begin{equation*}
    (p-1)|\vp'|^{p-2} \vp'' + (n-1)\sqrt{-k}\tanh\left(\sqrt{-k}t\right)|\vp'|^{p-2} \vp' =-\mu_{p}(n,k,D) |\vp|^{p-2} \vp
\end{equation*}
on $[-D/2,D/2]$.
\end{thm}
Matei\cite{Matei00} proved part (i) of the above theorem, generalizing the well-known Lichnerowicz-Obata theorem(see for example\cite[Theorem 5.1]{Libook}) to the case of $p$-Laplacian. Part (ii) was due to Valtorta\cite{Valtorta12} and the equality occurs if and only if $M$ is a one-dimensional circle (when $M$ has no bournday) or a line segment of length $D$ (when $M$ has boundary).
Part (iii) was proved by Naber and Valtorta\cite{NV14}.
The lower bound in part (iii) is never attained, but it is sharp in the sense that one can build a sequence of Riemannian manifolds $M_i$ with $\Ric \geq (n-1)k$ and $diam(M_i) \to D$ such that $\lim_{i\to \infty} \l_{p}(M_i) =\mu_{p}(n,k,D)$. Geometrically, the $M_i$'s collapse to the one-dimensional interval $[-D/2,D/2]$.

It is a natural question to extend the above-mentioned theorems to the case of weighted $p$-Laplacian. Some non-sharp lower bounds for $\l_{p,f}$ have already been obtained by Wang\cite{Wang12} and by Wang and Li\cite{WL16}.
The aim of this work is to prove sharp lower bound estimates for the first nonzero eigenvalue of the weighted $p$-Laplacian operator on Bakry-Emery manifolds.
Our main theorem states
\begin{thm}\label{Thm Main}
Let $(M^n,g,f)$ be a compact Bakry-Emery manifold (possibly with smooth strictly convex boundary) with diameter $D$ and $\Ric +\nabla^2 f \geq \kappa \, g$ for $\kappa \in \R$.
Let $\l_{p,f}$ be the first nonzero eigenvalue of the weighted $p$-Laplacian $\Delta_{p,f}$ (with Neumann boundary conditions if $\partial M \neq \emptyset$).
Assume that either $1<p\leq 2$ or $\kappa \leq 0$,
then we have
\begin{equation*}
    \lambda_{p,f} \geq \mu_p(\kappa,D),
\end{equation*}
where $\mu_p(\kappa,D)$ is the first nonzero Neumann eigenvalue of the corresponding one-dimensional problem
\begin{equation}\label{ODE}
    (p-1)|\vp'|^{p-2}\vp'' - \kappa\, t \, |\vp'|^{p-2}\vp' =-\mu_p(\kappa,D) |\vp|^{p-2}\vp
\end{equation}
on $[-D/2,D/2]$
\end{thm}

When $p=2$, this theorem covers the result of Bakry and Qian mentioned in Theorem \ref{Thm BQ} for $\kappa \leq 0$.
When $\kappa=0$ and $f\equiv 0$, the theorem reduces to the result of Valtorta\cite{Valtorta12}. This also demonstrates the sharpness of the lower bound when $\kappa=0$ for any $1<p<\infty$, as it is achieved when $M$ is a one-dimensional circle of diameter $D$ or the line segment $[-D/2, D/2]$, with $f\equiv 0$. For $\kappa <0$, the sharpness can be shown by constructing for each $\e >0$, a Bakry-Emery manifold $(M,g,f)$ with diameter $D$ and $\l_{p,f} \leq \mu_p(\kappa, D) +\e$. This is a slight modification of the construction in\cite{AN12}, with the only difference being replacing the ODE \eqref{ODE p=2} with the ODE \eqref{ODE}.

We conjecture that the above theorem holds for $p >2$ and $\kappa >0$ as well.
In this case,
we provide a non-sharp lower bound, which extends \cite[Theorem 3.2]{Matei00} and improves \cite[Theorem 1.1]{WL16}. 
\begin{thm}\label{Thm Main 2}
Let $(M,g,f)$ and $\l_{p,f}$ be the same as in Theorem \ref{Thm Main}. Suppose that $p\geq 2$ and $\kappa >0$, then
\begin{equation*}
    \l_{p,f} \geq \left(\frac{\kappa}{p-1} \right)^{\frac{p}{2}}.
\end{equation*}
\end{thm}


We conclude this section by discussing the strategy of our proof and the organization of this paper.

In Section 2, we give a simple proof of Theorem \ref{Thm Main} for the case $1<p\leq 2$ using the
modulus of continuity estimates method of Andrews and Clutterbuck. This effective approach and its elliptic version has been successfully used in proving sharp lower bounds of the first nonzero eigenvalue\cite{AC13}, and the fundamental gap conjecture\cite{AC11}(see also \cite{Ni13} for an elliptic proof), and the fundamental gap conjecture for convex domains in the sphere\cite{SWW19}.






In Section 3, we prove sharp height-dependent gradient estimates for eigenfunctions of $\Delta_{p,f}$.
Such sharp gradient estimates are key ingredients in obtaining sharp eigenvalues estimates, as demonstrated by Bakry-Qian\cite[Section 5]{BQ00} for the $f$-Laplacian, Valtorta\cite[Theorem 4.1]{Valtorta12}, and Naber and Valtorta\cite[Theorem 16]{NV14} for the $p$-Laplacian. Our approach via the two-point maximum principle developed by Andrews\cite[Section 7]{Andrewssurvey15} is quite different from the approaches in \cite{BQ00}\cite{NV14}\cite{Valtorta12}. The wonderful survey\cite{Andrewssurvey15} discusses more for the application of the maximum principle functions depending on several points or to functions depending on the global structure of the solutions.

Section 4 is devoted to investigating the qualitative behavior of the one-dimensional model equation \eqref{ODE}.
We shall prove that for every eigenfunction of $\Delta_{p,f}$, there always exists an interval $[a,b]$ among all intervals which has the same eigenvalue as $u$ for the one-dimensional model, and such that the corresponding eigenfunction has the same range as the eigenfunction $u$.

With all the preparations in Sections 3 and 4, we prove Theorem \ref{Thm Main} for the case $\kappa \leq 0$ in Section 5.

In Section 6, we prove the non-sharp lower bound in Theorem \ref{Thm Main 2}. The proof uses only the Bochner formula for the $f$-Laplacian and integration by parts.


\section{The case $1<p \leq 2$}
In this section, we apply the modulus of continuity estimates to prove the sharp eigenvalue estimates for $1<p\leq 2$. The proof presented below is a modification of the argument outlined in the survey by Andrews\cite[Section 8]{Andrewssurvey15} for the special case $f\equiv 0$.
Recall that given a continuous function $u: M\rightarrow \R$,  $w$ is a modulus of continuity for $u$ if for all $x$ and $y$ in $M$,
$$
|u(y)-u(x)|\le 2 w\left(\frac{d(x,y)}{2}\right),
$$
where $d$ is the induced distance function on $(M,g)$.

\begin{thm}\label{thmb1}
Let $1<p \leq 2$.  Let $(M,g,f)$ be a compact Bakry-Emery manifold (possibly with smooth strictly convex boundary) with diameter $D$ and $\Ric +\nabla^2 f \geq \kappa \, g$ for some $\kappa \in \R$. Let $v(x,t):M\times [0,+\infty)\rightarrow \R$ be a $C^{2,1}$ solution to
\begin{equation}\label{eqb2}
\frac{\p v}{\p t}=|\Delta_{p,f} v|^{-\frac{p-2}{p-1}}\Delta_{p,f} v
\end{equation}
with Neumann boundary conditions if $\partial M$ is not empty. Suppose $v(x,0)$ has a modulus of continuity
$\vp_0(s):[0, \frac{D}{2}]\rightarrow \R$ with $\vp_0(0)=0$ and $\vp_0'(s)>0$ on $[0,\frac{D}{2}]$.
Assume further that there exists a function $\vp(s,t):[0,\frac D 2]\times \R_+\rightarrow \R$, satisfying the following properties:
\begin{itemize}
\item[(1)] $\vp(s,0)=\vp_0(s)$ on $[0,\frac D 2]$;
\item [(2)]$0\ge \frac{\p \vp}{\p t}\ge |\mathcal{L}_{p,\kappa}(\varphi)|^{-\frac{p-2}{p-1}}\mathcal{L}_{p,\kappa}(\vp) $ on $[0,\frac D 2]\times \R_+$;
\item [(3)] $\vp'(s,t)>0$ on $[0,\frac D 2]$;
\item [(4)] $\vp(0,t)\ge 0$ for each $t>0$.
\end{itemize}
Here $\vp'=\frac{\p}{\p s}\vp(s,t)$, $\vp''=\frac{\p^2}{\p s^2}\vp(s,t)$, and $\mathcal{L}_{p,\kappa}(\vp)=(p-1)|\vp'|^{p-2}\vp''- \kappa \, s|\vp'|^{p-2}\vp'$.
Then $\varphi(s,t)$ is a modulus of the continuity of $v(x,t)$ for all $t>0$.
\end{thm}

We state and prove an elementary lemma that will be used in the proof. 
\begin{lemma}\label{lmb1}
Let $f:\R \to \R$ be an increasing and odd function. Assume that $f(t)$ is convex for $t>0$. Then for any $\delta\ge 0$ and $t\in \R$, we have
\begin{equation}\label{eqb1}
f(t-2\delta)-f(t)\le -2f(\delta).
\end{equation}
\end{lemma}
\begin{proof}
We have $f(0)=0$ since $f$ is odd. The convexity implies for $a, b > 0$, $$f(a)=f\left(\frac{a}{a+b} (a+b)+\frac{b}{a+b} 0 \right) \leq \frac{a}{a+b}f(a+b).$$
Similarly, $$f(b) \leq \frac{b}{a+b}f(a+b),$$
and we conclude that for all $a,b \geq 0$, it holds
$$f(a)+f(b) \leq f(a+b). $$  
Case 1: $t\ge 2 \delta$. Then
$$
f(t-2\delta)+2f(\delta) \le f(t-2\delta) +f(2\delta) \le f(t).
$$
Case 2: $0\le t<2\delta$. Then by oddness and the convexity,
$$
f(t)-f(t-2\delta) =f(t)+f(2\delta-t)\ge 2 f \left( \frac{t+(2\delta-t)}{2}\right)=2f(\delta).
$$
Case 3: $t<0$. Then
$$
2f(\delta)+f(-t)\le f(2\delta) +f(-t) \le  f(2\delta-t).
$$
\end{proof}

\begin{proof}[Proof of Theorem \ref{thmb1}]
Consider
$$
C_\e(x,y,t) := v(y,t)-v(x,t)-2\vp\left(\frac{d(x,y)}{2},t\right)-\e e^t.
$$
It suffices to show that $C_\e\le 0$ for any $\e>0$.
We argue by contradiction and assume that there exists $(x_0,y_0,t_0)$ such that $C_\e$ reaches zero for the first time. In other words, $C_\e(x,y,t)$ attains its maximum zero on $M \times M \times [0,t_0]$ at $(x_0,y_0,t_0)$.
Clearly $x_0\neq y_0$. The strict convexity of $\partial M$, the Neumann condition   and the positivity of $\vp'$ rule out the possibility that $(x_0, y_0) \in \p (M\times M)$.
By the first derivative test, the time derivative inequality yields
$$
v_t(y,t)-v_t(x,t)-2\vp_t-\e e^t\ge 0
$$
at  $(x_0, y_0, t_0)$.  Using the equation \eqref{eqb2}, we have
\begin{equation}\label{eqb3}
\left.|\Delta_{p,f} v|^{-\frac{p-2}{p-1}}\Delta_{p,f} v\right|_{(y_0, t_0)}-\left.|\Delta_{p,f} v|^{-\frac{p-2}{p-1}}\Delta_{p,f} v\right|_{(x_0, t_0)}-2\vp_t(d_0,t_0)\ge \e e^{t_0},
\end{equation}
where $d_0=\frac{d(x_0,y_0)}{2}$.
The positivity of $\vp'$  also gives that for any $t\le t_0$,
\begin{equation}\label{eqb4}
   v(\gamma(d),t) -v(\gamma(-d),t) - 2 \vp\left(\frac{L[\gamma]}{2},t \right) -\e e^t \leq 0
\end{equation}
for  any smooth path $\gamma:[-d,d] \to M$, where $L[\gamma]$ denotes the length of $\gamma$.
Moreover, the equality in \eqref{eqb4} holds when $t=t_0$ and $\gamma =\gamma_0$, a unit speed minimizing geodesic from $x_0$ to $y_0$. This allows us to deduce inequalities from the first and second variations along any smooth family of curves passing through $\gamma_0$. Let $\gamma(r,s)$ be a family of smooth curves with $\gamma(0,s)=\gamma_0(s)$.
The first derivative at $r=0$ gives
\begin{equation}\label{eqb5}
    \nabla v (y_0,t_0) = \vp'(d_0,t_0) \gamma_0'(d_0), \text{  and  }
    \nabla v (x_0,t_0) =  \vp'(d_0,t_0) \gamma_0'(-d_0).
\end{equation}
Here and below, all derivatives of $\vp$ are evaluated at $(d_0, t_0)$.

To make the calculation of second variations easier, we introduce "Fermi coordinates" along the geodesic $\gamma_0$: choose an orthonormal basis $\{e_i\}_{i=1}^n$ at $x$ with $\gamma_0(-d_0) =e_n$ and then parallel transport along $\gamma_0$ to obtain an orthonormal basis $\{e_i(s)\}_{i=1}^n$ for $T_{\gamma_0(s)} M$ with $e_n(s) =\gamma_0'(s)$ for each $s \in [-d_0,d_0]$. For the variation $\gamma(r,s)= \gamma_0(s+\frac{rs}{d_0})$, we have
$$\left.\frac{\partial}{\partial r}\right|_{r=0} L[\gamma(r,s))]=2
\text{  and  } \left.\frac{\partial^2}{\partial r^2}\right|_{r=0} L[\gamma(r,s))] =0.$$
Thus the second derivative test yields
\begin{equation}\label{eqb6}
     v_{nn}(y_0,t_0)-v_{nn}(x_0,t_0) - 2\vp'' \leq 0,
\end{equation}
where the subscripts now denote covariant derivatives in directions corresponding to the basis $\{e_i\}$.
For $1\leq i \leq n-1$, the variation $\gamma(r,s) =\exp_{\gamma_0(s)} \left(r e_i(s) \right)$ has
$$\left.\frac{\partial}{\partial r}\right|_{r=0} L[\gamma(r,s))]=0, $$
and
$$
    \left.\frac{\partial^2}{\partial r^2}\right|_{r=0} L[\gamma(r,s))] = -\int_{-d_0}^{d_0} R(e_i,e_n,e_i,e_n) ds.
$$
Then the second derivative inequality produces
\begin{equation*}
    v_{ii}(y_0,t_0)-v_{ii}(x_0,t_0) +\vp'\int_{-d_0}^{d_0} R(e_i,e_n,e_i,e_n) ds \leq 0.
\end{equation*}
We then get, by summing over $1\leq i\leq n-1$,
\begin{equation}\label{eqb7}
    \sum_{i=1}^{n-1} \left(v_{ii}(y_0,t_0)-v_{ii}(x_0,t_0) \right) +\vp'\int_{-d_0}^{d_0} \Ric(e_n,e_n) ds \leq 0.
\end{equation}
Combining inequalities \eqref{eqb6} and \eqref{eqb7}  together and plugging into \eqref{eqb5},
\begin{eqnarray*}
&& \Delta_{p,f} v(y_0,t_0) -\Delta_{p,f}  v(x_0,t_0) \\
&=& |\nabla v(y_0,t_0)|^{p-2} \left( (p-1)v_{nn}(y_0,t_0) +\sum_{i=1}^{n-1} v_{ii}(y_0,t_0) -\langle \nabla v(y_0,t_0), \nabla f(y_0) \rangle  \right) \\
&& - |\nabla v(x_0,t_0)|^{p-2} \left( (p-1)v_{nn}(x_0,t_0) +\sum_{i=1}^{n-1} v_{ii}(x_0,t_0) -\langle \nabla v(x_0,t_0), \nabla f(x_0) \rangle  \right) \\
&\leq&  |\vp'|^{p-2}
\left( 2(p-1) \vp'' -\vp' \int_{-d_0}^{d_0} \Ric(e_n,e_n) ds - \vp' \langle e_n(d_0), \nabla f(y_0) \rangle - \vp' \langle  e_n(-d_0), \nabla f(x_0) \rangle\right)   \\
&=&  |\vp'|^{p-2}
\left( 2(p-1)\vp'' -\vp'\int_{-d_0}^{d_0} \left(\Ric(e_n,e_n) +\nabla^2 f (e_n,e_n) \right)\, ds \right)   \\
&\leq&  |\vp'|^{p-2}
\left( 2(p-1)\vp'' -2\kappa \, d_0 \, \vp' \right).
\end{eqnarray*}
Therefore, we get
\begin{equation}\label{eqb8}
\Delta_{p,f} v(y_0,t_0)\le \Delta_{p,f}  v(x_0,t_0)+ 2\mathcal{L}_{p,\kappa}(\vp).
\end{equation}
Now let $h(t)=|t|^{-\frac{p-2}{p-1}}t$, which is increasing, odd, and convex for $t>0$.
We then deduce from \eqref{eqb8} that
$$
h(\Delta_{p,f} v(y_0,t_0))\le h(\Delta_{p,f}  v(x_0,t_0)+2\mathcal{L}_{p,\kappa}(\vp)).
$$
On the other hand, applying Lemma \ref{lmb1} with  $\delta=-\mathcal{L}_{p,\kappa}(\vp) \geq 0$ and
$t=\Delta_{p,f} v(x_0,t_0)$ yields
\begin{eqnarray*}
h(\Delta_{p,f} v(y_0,t_0))-h(\Delta_{p,f} v(x_0,t_0))\le h(t-2\delta)-h(t)\le -2h(\delta),
\end{eqnarray*}
which gives
\begin{equation}\label{eqb9}
\left.|\Delta_{p,f} v|^{-\frac{p-2}{p-1}}\Delta_{p,f} v\right|_{(y_0,t_0)}-\left.|\Delta_{p,f} v|^{-\frac{p-2}{p-1}}\Delta_{p,f} v\right|_{(x_0, t_0)}\le 2 |\mathcal{L}_{p,\kappa}(\vp)|^{-\frac{p-2}{p-1}}\mathcal{L}_{p,\kappa}(\vp).
\end{equation}
Finally, the two inequalities \eqref{eqb3} and \eqref{eqb9} yield the inequality
\begin{eqnarray*}
\vp_t\le |\mathcal{L}_{p,\kappa}(\vp)|^{-\frac{p-2}{p-1}}\mathcal{L}_{p,\kappa}(\vp)-\frac{\e}{2}e^{t_0}.
\end{eqnarray*}
This contradicts the inequality in assumption (2), and the proof is complete.
\end{proof}
\begin{remark}
From the above proof, we see that Theorem \ref{thmb1} holds when the function $v(x,t)$ is only $C^{1,1}$ in $(x,t)$ and $C^2$ in $x$ at points where $\nabla v \neq 0$.
\end{remark}

On the interval $[0, \frac{D}{2}]$, we define the following corresponding one-dimensional eigenvalue problem with boundary conditions
$\phi(0)=0$ and $\phi'(\frac{D}{2})=0$:
\begin{equation}
\lambda_{p,\kappa,D/2}=\inf \left\{\frac{\int_0^{\frac D 2} |\phi'|^p e^{-\frac \kappa 2 s^2} ds}{\int_0^{\frac D 2} |\phi|^p e^{-\frac \kappa 2 s^2} ds}, \text{\quad with\quad} \phi(0)=0 \right\}.
\end{equation}
\begin{lemma}\label{lmb2}
\begin{equation}
\mu_{p,\kappa,D}=\lambda_{p,\kappa,D/2}.
\end{equation}
\end{lemma}
\begin{proof}
Let $\phi(s)$ be an eigenfunction on $[0, D/2]$ corresponding to $\lambda_{p,\kappa,D/2}$. Then for
$s\in[-D/2, 0)$ we define $\phi(s)$ by $\phi(s)=-\phi(-s)$. Clearly, $\phi(s)$ defined on $[-D/2, D/2]$ is a trial function
for $\mu_{p,\kappa,D}$. Therefore
$$
\mu_{p,\kappa,D}\le \lambda_{p,\kappa,D/2}.
$$

On the other hand, Let $\psi(s)$ be an eigenfunction corresponding to $\mu_{p,\kappa,D}$. Without of loss of generality we assume further that $\psi(s_0)=0$ for some $s_0\in [0, D/2]$. If $s_0>0$, we define $\psi(s)=0$ for $s\in[0, s_0)$. Then
$\psi(s)$ is a trial function for $\lambda_{p,\kappa,D/2}$, and we have
$$
\lambda_{p,\kappa,D/2}\le\frac{\int_0^{\frac D 2} |\psi'|^p e^{-\frac \kappa 2 s^2} ds}{\int_0^{\frac D 2} |\psi|^p e^{-\frac \kappa 2 s^2} ds}=\frac{\int_{s_0}^{\frac D 2} |\psi'|^p e^{-\frac \kappa 2 s^2} ds}{\int_{s_0}^{\frac D 2} |\psi|^p e^{-\frac \kappa 2 s^2} ds}=\mu_{p,\kappa,D},
$$
proving the lemma.
\end{proof}

\begin{lemma}
There exists an odd eigenfunction  $\phi$ corresponding to $\mu_{p,\kappa,D}$ satisfying
$$
 (p-1)|\dot{\phi}|^{p-2}\ddot{\phi} - \kappa\, s \, |\dot{\phi}|^{p-2}\dot{\phi} +\mu_{p,\kappa,D} |\phi|^{p-2}\phi =0
$$
in $(0, D/2)$ with  $\phi'(s)>0$ in $(0, D/2)$ and $\phi'(D/2)=0$.

\end{lemma}
\begin{proof}
Lemma \ref{lmb2} gives the existence of odd eigenfunction $\phi$ corresponding to $\mu_{p,\kappa,D}$, and $\phi\neq0$ in
$(0,D/2]$. Then we can assume $\phi(s)>0$ in $(0,D/2]$. Since
$$
\Big(e^{-\frac{\kappa}{2}s^2}|\dot{\phi}|^{p-2}\dot{\phi}\Big)'=-\mu_{p,\kappa,D} |\phi|^{p-2}\phi<0
$$
in $(0,D/2]$ and $\phi'(D/2)=0$, then $\phi'(s)>0$ in $(0, D/2)$.

\end{proof}

Now we turn to prove Theorem \ref{Thm Main} for the case $1<p\le2$.
\begin{proof}[Proof of Theorem \ref{Thm Main}]
For any $D_1>D$, let $\phi(s)$ be an eigenfunction corresponding to $\mu_{p,\kappa,D_1}$ with initial conditions $\phi(0)=0$ and $\phi'(s)>0$ on $(0,D/2]$. Let $u(x)$ be an eigenfunction  corresponding to $\lambda_{p,f}$.
Consider
$$
v(x,t)=e^{- t\lambda_{p,f}^{\frac{1}{p-1}}}u(x),
$$
and
$$
\vp(s,t)= c\, e^{-t\mu_{p,\kappa,D_1}^{\frac{1}{p-1}}}\phi(s),
$$
where $c$ is a positive constant so chosen that
$$
u(y)-u(x)-2c \, \phi\left(\frac{d(x,y)}{2}\right)\le 0.
$$
Then direct calculations show
$$
\frac{\p v}{\p t}=|\Delta_{p,f} v|^{-\frac{p-2}{p-1}}\Delta_{p,f} v,
$$
and
$$ \frac{\p \vp}{\p t}=|\mathcal{L}_{p,\kappa}(\varphi)|^{-\frac{p-2}{p-1}}\mathcal{L}_{p,\kappa}(\vp).$$
Moreover, it's easy to verify that $\vp$ satisfies all the conditions in Theorem 2.1. Then
\begin{equation}\label{eqb12}
 e^{- t\lambda_{p,f}^{\frac{1}{p-1}}} (u(y)-u(x))\le2 c\,  e^{-\mu_{p,\kappa,D_1}^{\frac{1}{p-1}}t}\phi\left(\frac{d(x,y)}{2}\right)
\end{equation}
for any $t>0$. Thus as $t\rightarrow \infty$, inequality \eqref{eqb12} implies
$$
\lambda_{p,f}\ge \mu_{p,\kappa,D_1}.
$$
Then Theorem \ref{Thm Main} follows by letting $D_1\rightarrow D$.
\end{proof}
\begin{remark}\label{rmk regularity}
The eigenfunction $u$ is in general not smooth. We have $u \in C^{1,\a}(M) \cap W^{1,p}(M)$, and elliptic theory ensures that $u$ is smooth where $\nabla u\neq 0$ and $u\neq 0$.
If $\nabla u \neq 0$ and $u(x)=0$, then $u \in C^{3,\a}(U)$ if $p>2$ and $u \in C^{2,\a}(U)$ if $1<p<2$, where $U$ is a small neighborhood of $x$.
We refer the reader to \cite{Tolksdorf84} for these results.
\end{remark}

\section{Sharp Gradient Estimates of Eigenfunctions}


There seems to be no way to make the proof via the modulus of continuity estimates in the previous section work for $p>2$.
Therefore, we use the classical gradient estimates method to deal with $p>2$.
The approach 
were successfully used to obtain optimal lower bounds for the first nonzero eigenvalue (see \cite{Kroger98} for the Laplacian, \cite{BQ00} for the $f$-Laplacian, and \cite{Valtorta12}\cite{NV14} for the $p$-Laplacian).

The goal of this section is to prove sharp gradient estimates for eigenfunctions of $\Delta_{p,f}$. We use the two-point maximum principle to establish height-dependent gradient estimates, which is much easier than the classical approach used \cite{Kroger98}\cite{BQ00}\cite{Valtorta12}\cite{NV14}. Recently, the two-point maximum principle method was used by Andrews and Xiong\cite{AX19} to derive gradient estimates for a wild class of non-singular isotropic quasi-linear elliptic equations. Moreover, it works in the low-regularity situations for viscosity solutions \cite{Li16}\cite{LW17}.

We state the sharp gradient comparison theorem for the eigenfunctions, which holds for all $1<p<\infty$ and $\kappa \leq 0$.
\begin{thm}\label{Thm gradient comparison}
Suppose that $(M,g)$ is a compact Riemannian manifold with $\Ric +\nabla^2 f \geq \kappa\, g$ for some $\kappa \leq 0$.
Suppose that $u$ is an eigenfunction of $\Delta_{p,f}$ with eigenvalue $\l$:
\begin{equation}
    \Delta_{p,f} u=-\l |u|^{p-2} u.
\end{equation}
Suppose $\phi:[a,b] \to \R$ is a solution of the one-dimensional equation:
\begin{equation}
    (p-1)|\phi'|^{p-2}\phi'' -\kappa\, t  |\phi'|^{p-2}\phi' =-\lambda |\phi|^{p-2} \phi
\end{equation}
which is increasing, and such that the range of $u$ is contained in $[\phi(a),\phi(b)]$. Let $\Psi$ be the inverse of $\phi$.
Then for every $x$ and $y$ in $M$, it holds
\begin{equation}
    \Psi(u(y)) -\Psi(u(x)) -d(x,y) \leq 0.
\end{equation}
\end{thm}

As an immediate corollary, we get the following sharp gradient comparison theorem by allowing $y$ to approach $x$,.
\begin{corollary}[Gradient Comparison]\label{Cor gradient comparion}
Under the assumptions of Theorem \ref{Thm gradient comparison}, we have
\begin{equation}
    |\nabla u(x)| \leq \phi'\left(\Psi(u(x))\right)
\end{equation}
for all $x\in M$.
\end{corollary}

\begin{proof}[Proof of Theorem \ref{Thm gradient comparison}]
First of all, it suffices to prove the theorem for $\kappa <0$, as the case $\kappa=0$ follows immediately by letting $\kappa \to 0$. Secondly, it suffices to consider the case $\phi' >0$ by approximation. Thirdly, we assume $\p M =\emptyset$ for a moment and address the difference of the proof when $\p M \neq \emptyset$ at the end.
Consider the continuous function $Z:M\times M \to \R$ defined by
\begin{equation}
    Z(x,y):=\Psi\left(u(y)\right) -\Psi\left(u(x)\right)-d(x,y)
\end{equation}
and let $m$ be its maximum on $M \times M$.
We shall derive a contradiction if $m>0$.
Since $Z$ vanishes on the diagonal of $M\times M$, its positive maximum must be attained at some $(x,y) \in M\times M$ with $x\neq y$.
This implies that
\begin{equation}
   \Psi\left(u(\gamma(1))\right) -\Psi\left(u(\gamma(0))\right)-L[\gamma] \leq m
\end{equation}
for all smooth curves $\gamma$ in $M$, with equality when $\gamma=\gamma_0:[0,d] \to M$, a unit speed minimizing geodesic from $x$ to $y$ with $d=d(x,y)$. We then pick Fermit coordinate $\{e_i(s)\}_{i=1}^n$ along $\gamma_0$ so that $e_n(s)=\gamma_0'(s)$.  For simplicity of notations, we write $z_x=\Psi\left(u(x)\right)$ and $z_y=\Psi\left(u(y)\right)$, or
equivalently, $\phi(z_x)=u(x)$ and $\phi(z_y)=u(y)$.
Varying the endpoints gives the first order identities:
\begin{equation}
    \nabla u(y) =  \phi'(z_y) \gamma_0'(d), \text{   }
     \nabla u(x) = \phi'(z_x) \gamma_0'(0). \text{   }
\end{equation}
In particular, we have $\nabla u(x) \neq 0$ and $\nabla u(y)\neq 0$. By Remark \ref{rmk regularity}, $u$ is at least $C^{2,\a}$ near $x$ and $y$.
The second variation in $x$ and $y$ along $\gamma_0$ produce two inequalities:
\begin{align}\label{eq3.7}
    & \frac{\partial^2 Z}{\partial x_n^2}=\frac{1}{\phi'(z_x)} \left(\phi''(z_x) -u_{nn}(x) \right) \leq 0, \\ \label{eq3.8}
    & \frac{\partial^2 Z}{\partial y_n^2}=\frac{1}{\phi'(z_y)} \left(u_{nn}(y)-\phi''(z_y) \right) \leq 0.
\end{align}
In the transverse directions, we choose as before the variation $\gamma_v(s) =\exp_{\gamma_0} (v e_i(s)) $ with $1 \leq i \leq n-1$, yielding the following inequality:
\begin{equation}
    \frac{\partial^2 Z}{\partial x_i^2} + \frac{\partial^2 Z}{\partial y_i^2} =\frac{u_{ii}(y)}{\phi'(z_y)} -\frac{u_{ii}(x)}{\phi'(z_x)} +\int_{\gamma_0} R(e_i,e_n,e_i,e_n) \leq 0.
 \end{equation}
Adding these over $1\leq i\leq n-1$ gives,
\begin{equation}\label{eq3.9}
    \sum_{i=1}^{n-1}\left(\frac{\partial^2 Z}{\partial x_i^2} + \frac{\partial^2 Z}{\partial y_i^2} \right) =\sum_{i=1}^{n-1} \left(\frac{u_{ii}(y)}{\phi'(z_y)} -\frac{u_{ii}(x)}{\phi'(z_x)} \right)
   +\int_{\gamma_0} \Ric(e_n,e_n) \leq 0.
\end{equation}
The curvature assumption $\Ric +\nabla^2 f \geq \kappa\, g$ implies
\begin{align}\label{eq3.10} \nonumber
    \int_{\gamma_0} \Ric(e_n,e_n)
    & \geq \kappa \cdot d(x,y) - \int_{\gamma_0} \nabla^2 f(e_n,e_n) \\ \nonumber
    & \geq \kappa \cdot d(x,y) -\int_0^d \frac{d}{ds}\langle \nabla f, e_n \rangle ds \\ \nonumber
    & = \kappa \cdot d(x,y) -\langle \nabla f(y),\gamma_0'(1) \rangle + \langle \nabla f(x), \gamma_0'(0) \rangle \\
    &=\kappa \cdot d(x,y) -\frac{\langle \nabla f(y),\nabla u(y) \rangle}{\phi'(z_y)} + \frac{\langle \nabla f(x), \nabla u(x) \rangle}{\phi(z_x)}
\end{align}


Combining \eqref{eq3.7}, \eqref{eq3.8}, and \eqref{eq3.9} together, and using \eqref{eq3.10}, we obtain the following inequality:
\begin{eqnarray*}
0&\geq &(p-1)\frac{\partial^2 Z}{\partial x_n^2} +(p-1)\frac{\partial^2 Z}{\partial x_n^2}
+ \sum_{i=1}^{n-1}\left( \frac{\partial^2 Z}{\partial x_i^2}+\frac{\partial^2 Z}{\partial y_i^2} \right) \\
&=& \frac{1}{\phi'(z_y)} \left((p-1)u_{nn}(y) +\sum_{i=1}^{n-1} u_{ii}(y) -(p-1)\phi''(z_y)  \right)  \\
&&-\frac{1}{\phi'(z_x)} \left((p-1)u_{nn}(x) +\sum_{i=1}^{n-1} u_{ii}(x)  -(p-1)\phi''(z_x)  \right) +\int_{\gamma_0} \Ric(e_n,e_n) \\
&\geq& \frac{1}{\phi'(z_y)} \left((p-1)u_{nn}(y) +\sum_{i=1}^{n-1} u_{ii}(y) - \langle \nabla f(y),\nabla u(y) \rangle -(p-1)\phi''(z_y)  \right)  \\
&&-\frac{1}{\phi'(z_x)} \left((p-1)u_{nn}(x) +\sum_{i=1}^{n-1} u_{ii}(x) - \langle \nabla f(x),\nabla u(x) \rangle -(p-1)\phi''(z_x)  \right)  \\
&& +\kappa \cdot d(x,y) \\
&=& \frac{1}{\phi'(z_y)} \frac{\Delta_{p,f}u(y)}{|Du(y)|^{p-2}} -\frac{1}{\phi'(z_x)} \frac{\Delta_{p,f}u(x)}{|Du(x)|^{p-2}} -\frac{(p-1)\phi''(z_y)}{\phi'(z_y)} +\frac{(p-1)\phi''(z_x)}{\phi'(z_x)}  \\
&& +\kappa \cdot d(x,y) \\
&=& \frac{-\lambda |u(y)|^{p-2}u(y)}{\phi'(z_y)^{p-1}}-\frac{-\lambda |u(x)|^{p-2}u(x)}{\phi'(z_x)^{p-1}}  -\kappa z_y+\frac{\lambda |\phi(z_y)|^{p-2}\phi(z_y)}{ \phi'(z_y)^{p-1}} \\
&& + \kappa z_x -\frac{\lambda |\phi(z_x)|^{p-2}\phi(z_x)}{\phi'(z_x)^{p-1}}  +\kappa \cdot d(x,y)\\
&=& -\kappa \left(\Psi(u(y))-\Psi(u(x)) -d(x,y) \right)= -\kappa m >0,
\end{eqnarray*}
This is a contradiction since $\kappa <0$ and $m>0$, and we finish the proof when $\p M = \emptyset$.

For the case $M$ has boundary, the only difference in the proof is that the point $x$ or $y$ may lie on the boundary of $M$. This can be easily ruled out by the strict convexity of $\partial M$ and the Neumann boundary condition.  If $x\in \partial M$, then at $x$,
\begin{equation*}
   \frac{\partial}{\partial \nu} \left(\Psi(u(y))-\Psi(u(x))-d(x,y) \right) = - \frac{\partial}{\partial \nu}d(x,y) >0,
\end{equation*}
contradicting the strict convexity of $\p M $.
\end{proof}





\section{The One-dimensional equation}
In this section, we examine the one-dimensional equation \eqref{ODE}:
\begin{equation}\label{eq4.1}
     |w'|^{p-2}w'' -\kappa \, t |w'|^{p-2}w'+\lambda |w|^{p-2}w =0.
\end{equation}
For the purpose of getting sharp eigenvalue estimates, we need to show that for any eigenfunction $u$ of $\Delta_{p,f}$ with eigenvalue $\l >0$, there exist an interval $[a,b]$ and a solution $w$ of \eqref{eq4.1} such that $w$ is strictly increasing on $[a,b]$ with $w(a)=u_{\min}$ and $w(b)=u_{\max}$.
We achieve this by varing the initial data.

Fix $\l >0$ and $\kappa \leq  0$.
We consider the 
following initial value problem(IVP)
\begin{equation}\label{IVP}
    \begin{cases}
    |w'|^{p-2}w'' -\kappa \, t |w'|^{p-2}w'+\lambda |w|^{p-2}w =0, \\
    w(a)=-1, w'(a)=0,
    \end{cases}
\end{equation}
with $a\in \R$.
In proving the existence, uniqueness, and continuous dependence with respect to the parameters for the solutions of the initial value problem \eqref{IVP}, the fundamental role is
played by the generalized Pr\"ufer transformation introduced by Elbert\cite{Elbert79}.
To begin with, we recall some basic definitions and properties of $p$-trigonometric functions and refer the reader to \cite[Chapter 1]{DP05} for a more detailed study. For $1<p<\infty$, let $\pi_p$ be the positive number defined by
\begin{equation*}
    \pi_p=\int_{-1}^1 \frac{ds}{(1-s^p)^{1/p}} =\frac{2\pi}{\sin(\pi/p)}.
\end{equation*}
The $p$-sine function $\sin_p:\R \to [-1,1]$ is defined implicitly on $[-\pi_p/2,3\pi_p/2]$ by
\begin{equation*}
    \begin{cases}
    t=\int_0^{\sin_p(t)} \frac{ds}{(1-s^p)^{1/p}} &  \text{ if } t\in [-\frac{\pi_p}{2},\frac{\pi_p}{2}], \\
    \sin_p(t)=\sin_p(\pi_p-t) & \text{ if }  t\in [\frac{\pi_p}{2}, \frac{3\pi_p}{2}],
    \end{cases}
\end{equation*}
and is periodic on $\R$ with period $2\pi_p$. It's easy to see that for $p\neq 2$ this function is smooth around noncritical points, but only $C^{1,\alpha}(\R)$ with $\a =\min\{p-1,(p-1)^{-1} \}$.
By defining
\begin{equation*}
 \cos_p(t)=\frac{d}{dt} \sin_p(t) \text{ and } \tan_p(t) =\frac{\sin_p(t)}{\cos_p(t)},
\end{equation*}
we then have the following generalized trigonometric identities:
\begin{align*}
   & |\sin_p(t)|^p+|\cos_p(t)|^p =1, \\
   &    \frac{d}{dt}\tan_p(t) =\frac{1}{|\cos_p(t)|^p} =1+|\tan_p(t)|^p, \\
   & \frac{d}{dt}\arctan_p(t) =\frac{1}{1+|t|^p}.
\end{align*}
Let $\alpha=\left(\frac{\lambda}{p-1}\right)^{p-1}$.
We introduce the $p$-polar coordinates $r$ and $\theta$ defined by
\begin{equation}
    \alpha w =r \sin_{p}(\theta), \; w'=r \cos_p(\theta),
\end{equation}
or equivalently,
\begin{equation}
    r=\left( (w')^p +\alpha^p w^p \right)^{\frac 1 p}, \; \theta=\arctan_p\left(\frac{\alpha w}{w'}\right).
\end{equation}
If $w$ is a solution of \eqref{IVP}, then direct calculation shows that $\theta$ and $r$ satisfy
\begin{align}\label{IVP 3}
    & \begin{cases}
    \theta'=\alpha -\frac{\kappa t}{p-1} \cos_p^{p-1}(\theta)\sin_p(\theta), \\
    \theta(a) = -\frac{\pi_p}{2}, \; ( \text{mod} \, \pi_p);
    \end{cases} \\ \label{IVP 4}
    & \begin{cases}
    \frac{d}{dt}\log r=\frac{\kappa t}{p-1} \cos_p^{p}(\theta),  \\
    r(a) = \alpha.
    \end{cases}
\end{align}
Since both $\sin_p(t)$ and $\cos_p^{p-1}(t)$ are Lipschitz functions with Lipschitz constant $1$, we can apply Cauchy's theorem to obtain existence, uniqueness, and continuous dependence on the parameters. Indeed, we have the following proposition.
\begin{prop}
For any $a \in \R$, there exists a unique solution $w$ to \eqref{IVP} defined on $\R$ with $w, (w')^{p-2}w' \in C^1(\R)$. Moreover, the solution depends continuouly on the parameters in the sense of local uniform covergence of $w$ and $\dot{w}$ in $\R$.
\end{prop}

\begin{remark}
For the case $\kappa =0$, equation \eqref{eq4.1} can be explicitly solved.
By \eqref{IVP 4}, we have $r(t)= \left( w'^p +\alpha^p w^p \right)^{\frac 1 p} \equiv \a$. This implies, by separation of variables, that $w(t)=A \sin_p(\l^{\frac 1 p} x +B)$ for some constants $A$ and $B$. The initial conditions $w(a)=-1$ and $w'(a)=0$ determines $A$ and $B$.
\end{remark}

Next, we investigate the qualitative behavior of solutions to \eqref{IVP} when $\kappa <0$.
\begin{prop}\label{prop4.2}
Fix $\alpha>0$ and $\kappa <0$. There exists a unique $\bar{a}>0$ such that the solution $w_{-\bar{a}}$ to the IVP \eqref{IVP} is odd. In particular, $w_{-\bar{a}}$ restricted to $[-\bar{a},\bar{a}]$ has nonnegative derivative and has maximum value equal to one.
\end{prop}

\begin{proof}
Consider the IVP
\begin{equation}\label{IVP 5}
    \begin{cases}
    \dot{\theta}=\alpha -\frac{\kappa t}{p-1} \cos_p^{p-1}(\theta)\sin_p(\theta), \\
    \theta(0) = 0;
    \end{cases}
\end{equation}
If $\theta(t)$ is a solution, then so is the function $-\theta(-t)$ with the same initial data. The uniqueness of solutions implies $\theta(t)=-\theta(-t)$, so $\theta(t)$ is an odd function.
It follows from $\kappa < 0$ that $\dot{\theta} \geq \alpha$ whenever $\theta \in [-\pi_p/2,\pi_p/2]$. Thus there exists $-\bar{a} \in (-\pi_p/(2\alpha), 0)$ such that $\theta(-\bar{a})=-\pi_p/2$.
It's easily seen that the corresponding solution $r(t)$ to \eqref{IVP 4} is even, regardless of its initial value.
The proposition follows by translating obtained information on $\theta$ and $r$ back to $w$.
\end{proof}


For the solution $w$ of \eqref{IVP}, we define
\begin{align*}
    b(a) &=\inf \{ b>a : \dot{w}(b)=0 \}, \\
   m(a) &= w_a (b(a)), \\
   \delta(a) & =b(a)-a.
\end{align*}
In other words, $b(a)$ is the first value $b>a$ such that $\dot{w}(b)=0$ with the convention that $b(a)=\infty$ if such a value does not exist.  Thus $w$ is strictly increasing on the interval $[a, b(a)]$ and $m(a)$ is the maximum of $w$ on $[a,b(a)]$.
The function $\delta(a)$ measures the length of the interval where $w(a)$ increases from $-1$ to $m(a)$.

We are concerned with the range of the function $m(a)$ as $a$ varies.
Since $w$ is an eigenfunction for the Neumann problem $ |w'|^{p-2}w'' -\kappa \, t |w'|^{p-2}w'+\lambda |w|^{p-2}w =0 $
with eigenvalue $\l$ on $[a,b(a)]$, we have
$$0= \int_a^{b(a)} \frac{d}{dt} \left(e^{-\frac{\kappa}{2}t^2} |w'(t)|^{p-2} w'(t) \right) dt=\int_a^{b(a)} -\l |w(t)|^{p-2} w(t) e^{-\frac{\kappa}{2}t^2} dt$$
and therefore $m(a) >0$.
By Proposition \ref{prop4.2}, we have $m(-\bar{a})=1$.   We shall show in the next proposition that $m(a)$ goes to zero as $a$  goes to $\infty$. It then follows from the Intermediate Value Theorem that the range of $m(a)$ covers $(0,1]$ when $a$ varies in $[-\bar{a}, \infty)$.
\begin{prop}\label{prop4.3}
$$\lim_{a \to \infty} m(a) =0. $$
\end{prop}
\begin{proof}
For $a > -\bar{a}$, consider the IVP
\begin{equation}\label{IVP 6}
    \begin{cases}
    |w'|^{p-2}w'' -\kappa \, t |w'|^{p-2}w'+\lambda |w|^{p-2}w =0, \\
    w(a)=-1, w'(a)=0,
    \end{cases}
\end{equation}
and the associated IVP
\begin{align}\label{IVP 7}
    & \begin{cases}
    \theta'=\alpha -\frac{\kappa t}{p-1} \cos_p^{p-1}(\theta)\sin_p(\theta), \\
    \theta(a) = -\frac{\pi_p}{2};
    \end{cases} \\ \label{IVP 8}
    & \begin{cases}
    \frac{d}{dt}\log r=\frac{\kappa t}{p-1} \cos_p^{p}(\theta),  \\
    r(a) = \alpha.
    \end{cases}
\end{align}
At $b(a)$, we have $w'=0$, which implies $\theta =\frac{\pi_p}{2}$ and $m(a) = w = \a\, r \sin_p(\theta) =\a \, r$.
In view of \eqref{IVP 8}, we have
\begin{equation}
    \log r(b(a)) -\log \a =\int_a^{b(a)} \frac{\kappa \, t}{p-1} \cos_p^p(\theta) dt.
\end{equation}
So it suffices to show $\lim_{a \to \infty} \int_a^{b(a)} t  \cos_p^p(\theta) dt  = \infty$.

Since $\theta$ increases from $-\pi_p/2$ to $\pi_p/2$ on the interval $[a, b(a)]$, there exists a unique $t_0\in (a, b(a))$ such that $\theta(t_0) =0$. On the interval $[a, t_0]$, we have $\theta' \leq \a $, which implies $t_0 \geq a +\pi_p/(2\a)$. For $\eps >0$ sufficiently small, if we let $a_\eps$ be the point in $(a, t_0)$ with $\theta (a_\eps) =-\pi_p/2 +\eps$, then we have
\begin{equation*}
\int_a^{b(a)} t  \cos_p^p(\theta) dt \geq \int_{a_\eps}^{t_0} t  \cos_p^p(\theta)  dt \geq
\frac{1}{2}\cos_p^p(-\pi_p/2 +\eps) (t_0 -a_\eps)(t_0 +a_\eps)
\end{equation*}
The right hand side above goes to infinity as $a$ goes to infinity in view of  $t_0 -a_\eps \geq \left(\pi_p/2 -\eps \right)/(2\a)$.
\end{proof}

\begin{prop}\label{prop4.4}
$\delta (a) \geq \delta(-\bar{a})$ for all $a \geq -\bar{a}$ with strict inequality if $a \neq -\bar{a}$.
\end{prop}
\begin{proof}
The proof below is a modification of the proof of \cite[Proposition 43]{NV14}. The key is that the function $-\kappa t$ is odd and convex.

For $a \in (-\bar{a}, \infty)$, let
\begin{equation*}
    \psi_+ (t)=\theta_{a}(t), \text{    } \vp(t) =\theta_{-\bar{a}}(t), \text{     } \psi_- (t)=-\theta_{a}(-t),
\end{equation*}
We study these functions when their range is  $[-\pi_p/2,\pi_p/2]$, that is, on the interval on which these functions increase from $-\pi_p/2$ to $\pi_p/2$.
As in the proof of Proposition \ref{prop4.2}, it is easy to see that $\vp$ is an odd function, and $\psi_{-}$ is also a solution of \eqref{IVP 7}. By comparison, we always have $\psi_{-}(t) > \vp(t) > \psi_{+}(t)$.
Since these functions have positive derivatives, we can consider their inverse functions defined on $[-\pi_p/2,\pi_p/2]$,
$$h=\psi_{-}^{-1}, \text{     } s=\vp^{-1}, \text{     } g=\psi_{+}^{-1}.$$
The function $m(\theta)$ defined by
$$m(\theta) :=\frac{1}{2} \left(h(\theta)+g(\theta) \right)$$
is an odd function with
$$m(\pi_p/2) = \frac{1}{2} \left(h(\pi_p/2)+g(\pi_p/2) \right) = \frac{1}{2}\left(b(a)-a \right)=\frac{1}{2}\delta(a).
$$
Thus the desired estimate is equivalent to $m(\pi_p/2) > \bar{a}$. By symmetry, it suffices to consider the set $\theta \geq 0$, or equivalently $m\geq 0$. It is easy to compute that $m$ satisfies the following ODE:
\begin{equation*}
    2\frac{dm}{d\theta} =\frac{1}{\a - g f(\theta)}+\frac{1}{\a - h f(\theta)},
\end{equation*}
with $m(0)=0$, where $f(\theta)=\frac{\kappa}{p-1} \cos^p_p(\theta)\sin_p(\theta)$.
Since the function $z(t)=(\a -\b t)^{-1}$ is convex for $\a,\b \geq 0$, we have
\begin{equation}
    \frac{dm}{d\theta} \geq \frac{1}{\a - m f(\theta)}
\end{equation}
for all those values of $\theta$ such that both $g$ and $h$ are nonnegative. However, by symmetry, it is easily seen that this inequality holds also when one of the two is negative. Moreover, the inequality is strict if $\b >0$(i.e. if $\theta \in (0\,\pi_p/2)$) and if $g\neq h$.
By noticing the function $s(t)=\vp^{-1}(t)$ satisfies
$$\frac{ds}{d\theta} = \frac{1}{\a - s f(\theta)},$$
we conclude using a ODE comparison that
$$m(\pi_p/2) > s(\pi_p/2) =\bar{a},$$
and the desired claim follows immediately.
\end{proof}

At last, we study $\bar{\delta} :=\delta(-\bar{a}) =2\bar{a}$ as a function of $\l$, having fixed $p$ and $\kappa$.
It's easy to that $\bar{\delta}$ is a strictly decreasing function and so invertible. Thus we can define its inverse $\l(\delta)$, which is a continuous and decreasing function. Moreover, it can be characterized in the following equivalent way.
\begin{prop}\label{prop4.5}
For fixed $\kappa <0, 1<p<\infty$, we have that given  $\delta>0$, $\l$ is the first nonzero Neumann eigenvalue of the one-dimensional problem
\begin{equation*}
    (p-1)|w'|^{p-2}w'' -\kappa \, t |w'|^{p-2}w' =-\l |w|^{p-2}w
\end{equation*}
on $[-\delta/2,\delta/2]$.
\end{prop}



\section{Proof of Theorem \ref{Thm Main}}
After all the preparations in the previous two sections, we finally prove Theorem \ref{Thm Main} for $\kappa\le0$ in this section.
\begin{proof}[Proof of Theorem \ref{Thm Main}]
It suffices to prove the case $\kappa<0$, as the case $\kappa=0$ follows by letting $\kappa \to 0$.
Let $u$ be an eigenfunction of $\Delta_{p,f}$ associated to the eigenvalue $\l$. In view of
\begin{equation*}
    \int_M |u|^{p-2} u \, e^{-f} =0,
\end{equation*}
we can normalize $u$ so that $u_{\min} =-1$ and $u_{\max} \in (0,1]$.

By Proposition \ref{prop4.2} and \ref{prop4.3},
there exists an interval $[a,b]$ and a solution $w$ of \eqref{eq4.1} such that $w$ is strictly increasing on $[a,b]$ with $w(a) =-1 =u_{\min}$ and $w(b) = u_{\max} \in (0,1]$. Moreover, we have $\l=\l_{p}([a,b])$, the first nonzero eigenvalue of the Neumann problem $ |w'|^{p-2}w'' -\kappa \, t |w'|^{p-2}w'+\lambda |w|^{p-2}w =0$ on $[a,b]$.

Let $x$ and $y$ be such that $u(x)=\min_{M} u$ and $u(y)=\max_M u$.
By Theorem \ref{Thm gradient comparison}, we have
$$D \geq d(x,y) \geq \Psi(u(y)) -\Psi(u(x)) = \Psi(u_{\max}) -\Psi(u_{min}) = b-a = \delta(a) \geq \delta(\bar{a}),$$
where the last inequality is proved in Proposition \ref{prop4.4}.
This and Proposition \ref{prop4.5} yield immediately to the desired estimate.
\end{proof}
When $\kappa=0$, we have that
$$
\mu_p(0,D)=(p-1)\frac{\pi_p^p}{D^p}.
$$
Then we conclude from Theorem \ref{Thm Main} that
\begin{corollary}
Let $(M^n,g,f)$ be a compact Bakry-Emery manifold (possibly with smooth strictly convex boundary) with diameter $D$ and $\Ric +\nabla^2 f \geq0 $, and $1<p<\infty$. Then
\begin{equation*}
    \lambda_{p,f} \geq (p-1)\frac{\pi_p^p}{D^p}.
\end{equation*}
\end{corollary}

\section{A lower bound when $p \geq 2$ and $\kappa >0$}

In this section, we prove a non-sharp lower bound for $\l_{p,f}$ when $p\geq 2$ and $\kappa>0$ using merely the Bochner formula for the $f$-Laplacian and integration by parts.

\begin{thm}\label{Thm Lich lowe bound}
Fix $p \geq 2$. Let $(M,g,f)$ be a compact Bakry-Emery manifold (possibly with smooth strictly convex boundary) satisfying $\Ric+\nabla^2 f \geq \kappa \, g$ for some $\kappa >0$.
Then the first nonzero eigenvalue of the weighted $p$-Laplacian (with Neumann boundary condition if $\p M \neq \emptyset$), denoted by $\l_{p,f}$, admits the following lower bound:
\begin{equation}\label{lower bound}
    \l_{p,f} \geq \left(\frac{\kappa}{p-1} \right)^{\frac{p}{2}}.
\end{equation}
\end{thm}
In the case $f=\text{const}$, Theorem \ref{Thm Lich lowe bound} was due to Matei\cite[Theorem 3.2]{Matei00}. It's easy to see that \eqref{lower bound} is better than the lower bound $ \l_{p,f} \geq \frac{\kappa^{\frac{p}{2}}}{(p-1)^{p-1}}$ obtained by Wang and Li\cite[Theorem 1.1]{WL16} when $p\geq 2$.
%
\begin{proof}[Proof of Theorem \ref{Thm Lich lowe bound}]
We only present the proof for the case $M$ is closed here, as the proof extends easily to the case $M$ has convex boundary and $u$ satisfies the Neumann boundary condition. Recall that the Bochner formula for the $f$-Laplacian states
\begin{equation*}
    \frac{1}{2}\Delta_f |\nabla u|^2 =|\nabla^2 u|^2 +\langle \nabla u, \nabla (\Delta_f u) \rangle +\Ric_f (\nabla u, \nabla u),
\end{equation*}
where $\Ric_f=\Ric+\nabla^2 f$.
By integration by parts, the Bochner formula, and the fact that $p \geq 2$, we have for any $u \in C^{2}(M)$,
\begin{eqnarray*}
&&\int_M\Delta_{p,f}u \, \Delta_f u e^{-f} \\ 
&=& -\int_M |\nabla u|^{p-2} \langle \nabla u, \nabla (\Delta_f u) \rangle e^{-f}\\
&\geq & \int_M \Ric_f(\nabla u, \nabla u) |\nabla u|^{p-2}e^{-f} -\frac{1}{2} \int_M (\Delta_f |\nabla u|^2 ) |\nabla u|^{p-2}  e^{-f} \\
&\geq & \kappa \int_M |\nabla u|^p e^{-f} +\frac 1 2 \int_M \langle \nabla |\nabla u|^2, \nabla |\nabla u|^{p-2} \rangle e^{-f} \\
&= & \kappa \int_M |\nabla u|^p e^{-f} +\frac{p-2}{4} \int_M |\nabla u|^{\frac{p-4}{2}}\left| \nabla |\nabla u|^2 \right|^2 e^{-f} \\
&\geq & \kappa \int_M |\nabla u|^p e^{-f}
\end{eqnarray*}
On the other hand, the equation $\Delta_{p,f} u=-\l_{p,f} |u|^{p-2}u$ implies
\begin{eqnarray*}
&&\int_M\Delta_{p,f}u \, \Delta_f u e^{-f} \\
&=& -\l_{p,f}\int_M |u|^{p-2} u \, \Delta_f u \, e^{-f} \\
&= &(p-1)\l_{p,f} \int_M  |u|^{p-2} |\nabla u|^2 e^{-f} \\
&\leq & (p-1)\l_{p,f} \left(\int_M |u|^p e^{-f} \right)^{\frac{p-2}{p}} \left(\int_M  |\nabla u|^p e^{-f}\right)^{\frac{2}{p}}
\end{eqnarray*}
where the last inequality comes from H\"older's inequality. Putting the above two estimates together, we get
\begin{equation*}
    \kappa \int_M |\nabla u|^p e^{-f} \leq (p-1)\l_{p,f} \left(\int_M |u|^p e^{-f} \right)^{\frac{p-2}{p}} \left(\int_M  |\nabla u|^p e^{-f}\right)^{\frac{2}{p}},
\end{equation*}
which further implies
\begin{equation*}
   \l_{p,f} \geq \frac{\kappa}{p-1}  \left( \frac{\int_M |\nabla u|^p e^{-f}}{\int_M  |u|^p e^{-f}}\right)^{\frac{p-2}{p}} \geq \frac{\kappa}{p-1}(\l_{p,f})^{\frac{p-2}{p}}.
\end{equation*}
It then follows that $(\l_{p,f})^{\frac{2}{p}} \geq \frac{\kappa}{p-1}$, and the proof is complete.
\end{proof}


\bigskip

\bibliographystyle{plain}
\bibliography{ref}

\end{document}